\documentclass[12pt,letterpaper]{amsart}
\usepackage{geometry}
\usepackage{hyperref, authblk}
\usepackage{mathpazo}
\usepackage{amssymb,url,setspace,xcolor}
\newtheorem{theorem}{Theorem}

\newtheorem{proposition}{Proposition}
\newtheorem{lemma}{Lemma}
\newtheorem{corollary}{Corollary}

\newtheorem{remark}{Remark}
\onehalfspace

\title[Compactness and Singular points of Composition Operators]{Compactness and Singular Points of Composition Operators on Bergman Spaces}
\author{Timothy G. Clos}
\address[Timothy G. Clos]{Bowling Green State University, Department of Mathematics and Statistics, Bowling Green, Ohio}
\email{clost@bgsu.edu}

\begin{document}

\begin{abstract}
Let $\Omega\subset \mathbb{C}^n$ for $n\geq 2$ be a bounded pseudoconvex domain with a $C^2$-smooth boundary.  We study the compactness of composition operators on the Bergman spaces of smoothly bounded convex domains.  We give a partial characterization of compactness of the composition operator (with sufficient regularity of the symbol) in terms of the behavior of the Jacobian on the boundary.  We then construct a counterexample to show the converse of the theorem is false. \end{abstract}

 \maketitle
 
\section{Introduction}
Let $\Omega\subset \mathbb{C}^n$ be a bounded pseudoconvex domain.  Let $\mathcal{O}(\Omega)$ be the set of all holomorphic functions from $\Omega$ into $\mathbb{C}$.  Let $V$ be the Lebesgue volume measure on $\Omega$.  For $p\in [1,\infty)$ we define
\[A^p(\Omega):=\{f\in \mathcal{O}(\Omega):\int_{\Omega}|f|^p dV<\infty\}\] to be the $p$-Bergman space.  We denote the norm
as \[\|f\|_{p,\Omega}:=\left(\int_{\Omega}|f|^p dV\right)^{\frac{1}{p}}.\]  Most of this paper will deal with the $2$-Bergman space, which for brevity is denoted as the Bergman space.
Let $\phi:\Omega\rightarrow \Omega$ be holomorphic on $\Omega$.  That is, holomorphic in each coordinate function. Then we define the composition operator with symbol $\phi$ as
\[C_{\phi}(f)=f\circ \phi\] for all $f\in A^p(\Omega)$.  For Banach spaces $X$ and $Y$, we say a linear operator $T:X\rightarrow Y$ is compact if $T(\{x\in X:\|x\|<1\})$ is relatively compact in the norm topology on $Y$.  If $X$ is a Hilbert space then we can characterize compactness of linear operator $T:X\rightarrow X$ in terms of weakly convergent sequences.

\section{Some Background and Main Results}
Compactness of composition operators was studied on the unit disk in $\mathbb{D}$ in the article \cite{ghatagecompact}. Here, the authors of \cite{ghatagecompact} study the angular derivative of the symbol near the boundary and obtain a compactness result.  They then construct a counterexample to show that the converse of their theorem does not hold true.  The authors of 
\cite{ghatageclosed} studied the closed range property of composition operators on the unit disk.  Work on essential norm estimates and compactness of composition operators was studied on the ball in $\mathbb{C}^n$ and on the unit disk in $\mathbb{C}$ by \cite{CowenMacCluer} and \cite{HMW}.  On more general bounded strongly pseudoconvex domains in $\mathbb{C}^n$, \cite{cuckruhanseveral} studied the essential norm of the composition operator in terms of the behavior of the norm of the normalized Bergman kernel composed with the symbol.  Our approach is to use the idea of the 'generalized angular derivative' (the Jacobian) of the symbol and its behavior on the boundary.  As a preliminary result, we assume $\Omega\subset \mathbb{C}^n$ is a bounded pseudoconvex domain and $\phi$ a holomorphic self-map on $\Omega$.  If the range of $\phi$ is compactly contained in $\Omega$, then the composition operator with symbol $\phi$ is compact on the Bergman space of $\Omega$.  The main result result relates compactness of the composition operator $C_{\phi}$ to the behavior of $J(\phi)(p)$ for $p\in b\Omega$ for $C^2$-smooth bounded convex domain $\Omega$.  Then as a corollary we show that there is no surjective proper map $\phi:\Omega\rightarrow \Omega$ that is $C^1$-smooth up to $\overline{\Omega}$ so that $C_{\phi}$ is compact on $A^2(\Omega)$.  Then we show the converse of the main result is false.

\begin{theorem}\label{singular}
Let $\Omega\subset\mathbb{C}^n$ for $n\geq 2$ be a bounded convex domain with a $C^2$-smooth boundary. Let $\phi:=(\phi_1,\phi_2,...,\phi_n):\Omega\rightarrow \Omega$ be holomorphic in every coordinate function and $C^1$-smooth in each coordinate function on $\overline{\Omega}$.  Then if $C_{\phi}:A^2(\Omega)\rightarrow A^2(\Omega)$ is compact on $A^2(\Omega)$ then 
$\phi^{-1}(b\Omega)=\emptyset$ or $\phi^{-1}(b\Omega)$ consists of singular points. 
\end{theorem}

As a consequence of Theorem \ref{singular}, one can obtain the following.

\begin{corollary}\label{thm2}
Let $\Omega\subset\mathbb{C}^n$ for $n\geq 2$ be a bounded convex domain with a smooth boundary. Let $\phi:=(\phi_1,\phi_2,...,\phi_n):\Omega\rightarrow \Omega$ be holomorphic in every coordinate functions and $C^1$-smooth in each coordinate function on $\overline{\Omega}$.  If $\phi$ is a surjective proper map then $C_{\phi}$ is not compact on $A^2(\Omega)$.
\end{corollary}




\section{Preliminaries}

We let $J(\phi)(p)$ be the determinant of the complex Jacobian matrix of $\phi$ at point $p$.  Then $|J(\phi)(p)|^2$ is the determinant of the real Jacobian matrix of $\phi$ at point $p$.  We note that \[\phi(z_1,z_2,...,z_n)=(\phi_1(z_1,z_2,...,z_n), \phi_2(z_1,z_2,...,z_n),...,\phi_n(z_1,z_2,...,z_n))\] where \[\phi_j\in C^{1}(\overline{\Omega})\] and are holomorphic on $\Omega$ for every $j=1,2,...,n$.  By the smoothness of $b\Omega$, we can extend $\phi_j$ as a smooth function on $\mathbb{C}^n$ for $j\in \{1,2,...,m\}$, also called $\phi_j$.
If $|J(\phi)(p)|\neq 0$ for all $p\in b\Omega$, we use the compactness of $b\Omega$ and the inverse function theorem applied to $\phi$, to cover $b\Omega$ with finitely many balls $\{B(p_s,r_s)\}_{s=1,...,k}$ so that $p_s\in b\Omega$ for all $s=1,...,k$, $r_s>0$ for all $s\in \{1,2,...,k\}$, and the restriction $\phi|_{B(p_s,r_s)}$ is invertible with inverse $\psi_s$ for $s\in \{1,2,...,k\}$.  Then there exists $\delta_0>0$ so that 
\[U_{\delta_0}:=\{(z_1,z_2,...,z_n)\in \Omega:{dist}((z_1,z_2,...,z_n),b\Omega)<\delta_0\}\subset\subset \bigcup_{s\in \{1,2,...,k\}}B(p_s,r_s)\cap\Omega\]

\begin{lemma}\label{lem1}

For $U_{\delta_0}$ defined previously, the measure $d\mu:=\chi_{U_{\delta_0}}dV$ is reverse Carleson with respect to
the Lebesgue volume measure $dV$ on a bounded pseudoconvex domain $\Omega\subset \mathbb{C}^n$ for $n\geq 2$.  That is, for every $p\in [1,\infty)$ and $g\in A^p(\Omega)$, there exists $M_{p,\delta_0}>0$ so that \[\|g\|_{p,\Omega}\leq M_{p,\delta_0}\|g\|_{p,U_{\delta_0}}.\]

\end{lemma}

\begin{proof}

We consider the restriction operator $R_{\delta_0}:A^p(\Omega)\rightarrow A^p(U_{\delta_0})$.
By the identity principle for holomorphic functions, $R_{\delta_0}$ is injective.  And by Hartog's extension theorem (see \cite[1.2.6]{krantz}), $R_{\delta_0}$ is surjective.  Therefore, $R_{\delta_0}$ is invertible.  It is clear that $R_{\delta_0}$ is bounded.  Therefore, by the Open Mapping theorem,  $R_{\delta_0}$ has a bounded inverse.  Then there exists $M>0$ so that 

\[\|f\|_{p,\Omega}=\|(R_{\delta_0})^{-1}R_{\delta_0}f\|_{p,\Omega}\leq M\|R_{\delta_0}f\|_{p,U_{\delta_0}}=M\|f\|_{p,U_{\delta_0}}.\]

This shows that $d\mu$ is reverse Carleson.

\end{proof}

\begin{remark}
Lemma \ref{lem1} is true for bounded domains in $\mathbb{C}$.  Instead of Hartog's extension theorem one must use the mean value principle for holomorphic functions.

\end{remark}

\begin{lemma}\label{lem2}
Let $\Omega_1,\Omega_2\subset \mathbb{C}^n$ for $n\geq 2$ be bounded pseudoconvex domains.  Furthermore, assume there exists a biholomorphism $B:\Omega_1\rightarrow \Omega_2$ so that $B\in C^1(\overline{\Omega_1})$.  Suppose $\phi:=(\phi_1,\phi_2,...,\phi_n):\Omega_2\rightarrow \Omega_2$ is such that the composition operator $C_{\phi}$ is compact on $A^2(\Omega_2)$.  Then, $C_{B^{-1}\circ \phi\circ B}$ is compact on $A^2(\Omega_1)$.  
\end{lemma}

\begin{proof}
Let $g_j\in A^2(\Omega_1)$ so that $g_j\rightarrow 0$ weakly as $j\rightarrow \infty$.  We will use the fact that $g_j\rightarrow 0$ weakly in $A^2(\Omega_1)$ as $j\rightarrow \infty$ if and only if $\|g_j\|$ is a bounded sequence in $j$ and $g_j\rightarrow 0$ uniformly on compact subsets of $\Omega_1$.  This fact appears as \cite[lemma 3.5]{cuckruhanseveral}. Therefore, $\|g_j\|$ is uniformly bounded in $j$ and $g_j\rightarrow 0$ uniformly on compact subsets of $\Omega_1$.
Then define $h_j:=g_j\circ B^{-1}\in  A^2(\Omega_2)$.  Then using a change of coordinates, one can show $\|h_j\|$ is uniformly bounded in $j$ and $h_j\rightarrow 0$ uniformly on compact subsets of $\Omega_2$.  Therefore, by \cite[lemma 3.5]{cuckruhanseveral}, $h_j\rightarrow 0$ weakly as $j\rightarrow \infty$.  
Then we have,
\begin{align*}
&\|C_{B^{-1}\circ \phi\circ B}(g_j)\|_{2,\Omega_1}^2\\
&=\|h_j\circ \phi\circ B\|_{2,\Omega_1}^2\\
&\leq \sup\{|J(B^{-1})(z)|^2:z\in \Omega_2\}\|C_{\phi}(h_j)\|_{2,\Omega_2}^2
\end{align*}

This shows that $C_{B^{-1}\circ\phi\circ B}$ is compact on $A^2(\Omega)$.

\end{proof}

As an application of Lemma \ref{lem1} and Lemma \ref{lem2}, we have the following proposition.

\begin{proposition}\label{thm1}
  Let $\Omega\subset\mathbb{C}^n$ for $n\geq 2$ be a bounded pseudoconvex domain with a smooth boundary. Let $\phi:=(\phi_1,\phi_2,...,\phi_n):\Omega\rightarrow \Omega$ be holomorphic on $\Omega$ in every component function and $C^1$-smooth in each component function on $\overline{\Omega}$.  If the Jacobian of $\phi$ is non-vanishing at every point in $b\Omega$, then the composition operator $C_{\phi}:A^p(\Omega)\rightarrow A^p(\Omega)$ is bounded for all $p\in [1,\infty)$.

\end{proposition}

\begin{proof}
To show if $f\in A^p(\phi(\Omega))$ then $f\circ \phi\in A^p(\Omega)$, it suffices to show $f\circ \phi\in A^p(U_{\delta_0})$ and apply Hartog's extension theorem and identity principle.  
We have
\begin{align*} 
&\int_{U_{\delta_0}}|f\circ\phi|^p dV\\
&\leq \sum_{s=1}^k\int_{B(p_s,r_s)\cap\Omega}|f\circ\phi|^p dV\\
&=\sum_{s=1}^k\int_{\phi(B(p_s,r_s)\cap\Omega)}|f|^p |J(\psi_s)|^2dV\\
&\leq k\sup_{s=1,2,...,k}\sup_{\phi(B(p_s,r_s)\cap\Omega)}|J(\psi_s)|^2\|f\|^p_{p,\phi(\Omega)}<\infty\\
\end{align*}
Then the boundedness of $C_{\phi}$ follows from an application of Lemma \ref{lem1} to the above string of inequalities.  

\end{proof}

Next we will focus our attention on compactness of the composition operator with holomorphic symbol $\phi:\Omega\rightarrow \Omega$ and show that if the range of $\phi$ is compactly contained in $\Omega$, then $C_{\phi}$ is compact on $A^2(\Omega)$.  

\begin{proposition}
Let $\Omega\subset \mathbb{C}^n$ be a bounded pseudoconvex domain.  Suppose $\phi$ is a holomorphic self-map on $\Omega$ so that $\overline{\phi(\Omega) }\subset \Omega$.  Then $C_{\phi}$ is compact on $A^2(\Omega)$.  

\end{proposition}

\begin{proof}
To prove compactness of $C_{\phi}$, it suffices to show that the image of a weakly convergent sequence in $A^2(\Omega)$ is strongly convergent.  Let $\{g_j\}_{j\in \mathbb{N}}\subset A^2(\Omega)$ converge to $0$ weakly as $j\rightarrow \infty$.  Then by 
\cite[lemma 3.5]{cuckruhanseveral}, $\|g_j\|_{L^2(\Omega)}$ is bounded and $g_j\rightarrow 0$ uniformly on compact subsets of $\Omega$.  We let $\nu:=V\circ \phi^{-1}$ be the pullback measure.  Since $\overline{\phi(\Omega)}$ is compactly contained in $\Omega$, there exists a compact set $\widetilde{\Omega}\subset \Omega$ so that the support of $\nu$ is contained in $\widetilde{\Omega}$.  Thus
\[\int_{\Omega}|g_j\circ \phi|^2dV=\int_{\Omega}|g_j|^2d\nu=\int_{\widetilde{\Omega}}|g_j|^2d\nu.\]  Since $\widetilde{\Omega}$ is compact and $g_j\rightarrow 0$ uniformly on $\widetilde{\Omega}$ as $j\rightarrow \infty$, we have that $C_{\phi}g_j\rightarrow 0$ as $j\rightarrow \infty$ in norm.

\end{proof}

\section{Proofs of Main Results}

\begin{proof}[Proof of Theorem \ref{singular}]
Assume $C_{\phi}$ is compact.
Suppose $\phi^{-1}(b\Omega)\neq \emptyset$ and so let $p:=(p_1,p_2,...,p_n)\in \phi^{-1}(b\Omega)$.  Without loss of generality and appealing to Lemma \ref{lem2}, we may assume $\phi(p)=(0,0,...,0):=0$.  Furthermore, since $\Omega$ is convex, we may assume $\Omega\subset \{(z_1,z_2,...,z_n)\in \mathbb{C}^n: Re(z_1)>0\}$.
Now assume $|J(\phi)(p)|\neq 0$.  Since $\Omega$ has a $C^2$-smooth boundary, we may extend $\phi$ as a $C^1$ function on an open neighborhood of $\overline{\Omega}$.  That is, extend each component function $\phi_j$ as a $C^1$ function on a neighborhood $U_j$ of $\overline{\Omega}$.  Then $\phi$ has a $C^1$-smooth extension to $U:=\bigcap_{j\in \{1,2,...,n\}}U_j$.  See \cite[section 5.17]{AdamsFournier} and \cite[section VI]{Stein} for more details on the smooth extension of functions.  Then by the inverse function theorem, there exists $\varepsilon>0$ so that $\phi$ is a $C^1$-diffeomorphism on \[B(p,\varepsilon):=\{(z_1,z_2,...,z_n)\in\mathbb{C}^n: |z_1-p_1|^2+...+|z_n-p_n|^2<\varepsilon^2\}.\] and $\phi\in C^1(\overline{B(p,\varepsilon)})$.
Furthermore, we may assume $|J(\phi)|\neq 0$ on $\overline{B(p,\varepsilon)}$.  We define 
\[g_j(z_1,z_2)=\frac{\alpha_j}{z_1^{\beta_j}}\]
where $\beta_j=1-\frac{1}{j}$ and $\alpha_j$ chosen so that $\alpha_j\rightarrow 0$ as $j\rightarrow \infty$ and 
$\|g_j\|=1$ for all $j\in \mathbb{N}$.  The convexity of $\Omega$ allows us to construct this $g_j$ so that $g_j\in A^2(\Omega)$ for all $j\in \mathbb{N}$ by taking appropriate branch cuts.  That is, we may assume $\Omega\subset \{(z_1,z_2,...,z_n)\in \mathbb{C}^n:Re(z_1)>0\}$ by convexity of $\Omega$ and take the principle branch for $z_1^{-\beta_j}$.  We construct $\alpha_j$ as follows.  Let $R>0$ be chosen sufficiently large so that $\Omega\subset \{z_1\in \mathbb{C}: |z_1|<R\}\times...\times \{z_n\in \mathbb{C}:|z_n|<R\}$.
Then converting to polar coordinates we have.
\begin{align*}
&\int_{\Omega}|z_1^{-\beta_j}|^2 dV\\
&\leq (\pi R^2)^{n-1}\int_{0\leq \theta\leq 2\pi}\int_{0\leq r\leq R}r^{-2\beta_j+1}drd\theta<\infty\\
\end{align*}  
Thus $z_1^{-\beta_j}\in A^2(\Omega)$ for all $j\in \mathbb{N}$.  By convexity of $\Omega$, there exists
$S>0$, $(\lambda_1,...,\lambda_n)\in [0,2\pi)^n$ and $(\gamma_1,...,\gamma_n)\in [0,2\pi)^n$ so that 
$\gamma_j>\lambda_j$ for $j\in \{1,2,...,n\}$ and
\[\{z_1=r_1e^{i\theta_1}:0\leq r_1\leq S,\, \lambda_1\leq \theta_1\leq \gamma_1\}\times...\times \{z_n=r_ne^{i\theta_n}:0\leq r_n\leq S,\, \lambda_n\leq \theta_n\leq \gamma_n\}\subset \Omega.\]
Now we convert to polar coordinates and it is clear that \[\int_{\lambda_1\leq \theta\leq \gamma_1}\int_{0\leq r\leq S}r^{-2\beta_j+1}dr d\theta\rightarrow \infty\] as $j\rightarrow \infty$.  Thus we define 
\[\alpha_j:=\|z_1^{-\beta_j}\|_{\Omega}^{-1}.\]

Then one can show $g_j\rightarrow 0$ uniformly on compact subsets of $\Omega$ as $j\rightarrow \infty$.  Furthermore, $\|g_j\|_{\Omega}=1$ for all $j\in \mathbb{N}$.  Hence by \cite[lemma 3.5]{cuckruhanseveral}, $g_j\rightarrow 0$ weakly as $j\rightarrow \infty$.  Since $\phi(B(p,\varepsilon))$ is open, there exists $\alpha>0$ so that $B(0,\alpha)\subset \phi(B(p,\varepsilon))$.  Furthermore, one can show that there exists $\delta>0$ so that $\|g_j\|_{L^2(B(0,\alpha)\cap \Omega)}\geq \delta$ for all $j\in \mathbb{N}$.  By shrinking $\varepsilon>0$ if necessary, we may assume there exists an $M>0$ so that \[\inf\{|J({\phi}^{-1})(z_1,z_2)|^2:(z_1,z_2)\in \overline{\phi(B(0,\alpha)\cap \Omega)}\}\geq M.\]
Then we have
\begin{align*}
&\int_{B(p,\varepsilon)\cap \Omega}|g_j\circ \phi|^2 dV\\
&=\int_{\phi(B(p,\varepsilon)\cap \Omega)}|g_j|^2 |J({\phi}^{-1})|^2 dV\\
&\geq \inf\{|J({\phi}^{-1})(z_1,z_2)|^2: (z_1,z_2)\in \overline{\phi(B(p,\varepsilon)\cap \Omega)} \}\|g_j\|_{L^2(\phi(B(p,\varepsilon)\cap \Omega))}^2\\
&\geq M \|g_j\|^2_{L^2(B(0,\alpha)\cap \Omega)}\\
&\geq M\delta^2>0\\
\end{align*}

Thus $\|C_{\phi}g_j\|_{L^2(\Omega)}$ does not converge to $0$, a contradiction.  This implies $|J(\phi)(p)|=0$.  That is,
$\phi^{-1}(b\Omega)=\emptyset$ or consists of singular points.

\end{proof}

\begin{proof}[Proof of Corollary \ref{thm2}]
Assume $\phi:\Omega\rightarrow \Omega$ is a proper holomorphic map where $\phi:=(\phi_1,...,\phi_n)$ and $\phi_j\in C^{1}(\overline{\Omega})$ for $j\in \{1,2...,n\}$.  Furthermore, assume that $C_{\phi}$ is compact on $A^2(\Omega)$.  Since $\phi$ is proper, it is an open map (see \cite[pp.789]{JP}), and therefore is surjective (see \cite{Ho}).  Also, $\phi^{-1}(b\Omega)=b\Omega$,  so by Theorem \ref{singular}, the determinant of the complex Jacobian of $\phi$ is identically $0$ on $b\Omega$.  Since the determinant of the complex Jacobian is a holomorphic function on $\Omega$ and is continuous up to $\overline{\Omega}$, we have that $|J(\phi)|\equiv 0$ on $\Omega$.  Thus an application of Sard's theorem (see \cite[theorem II 3.1]{Sternberg}) implies that $\phi$ is not surjective, a contradiction.  
\end{proof}

\section{Example}

The converse of Theorem \ref{singular} is not true.  In fact, we construct a highly singular $C^{\infty}$-smooth map $\phi:B(0,1)\rightarrow B(0,1)$ so that $\phi(\overline{B(0,1)})\cap b B(0,1)$ consists of singular points but $C_{\phi}$ is not compact on $A^2(B(0,1))$.  This idea is made precise as this next example shows.

We let $\phi(z_1,z_2):=(z_1,0)$ be the projection onto the first coordinate.  Clearly $\phi$ is singular everywhere and maps the unit ball into the unit ball.  Furthermore, 
$\phi(\overline{B(0,1)})\cap b B(0,1)\neq \emptyset$.  We let $\beta_j=1-\frac{1}{j}$ and define
\[f_j(z_1,z_2):=\frac{\alpha_j}{(z_1-1)^{\beta_j}}\] where $\alpha_j\rightarrow 0$ as $j\rightarrow \infty$ is to be chosen later.  Also, the branch cut is taken away from $B(0,1)$.  We chose $R>0$ sufficiently large so that
$B(0,1)\subset \{z\in \mathbb{C}:|z-1|<R\}\times \{w\in \mathbb{C}:|w|<R\}:=D$.  Thus converting to polar coordinates we have 
\begin{align*}
&\int_{B(0,1)}|(z_1-1)^{-\beta_j}|^2 dV\leq \int_{D}|(z_1-1)^{-\beta_j}|^2 dV\\
&=\pi R^2\int_{0\leq r\leq R}\int_{0\leq \theta\leq 2\pi}r^{-2\beta_j+1}dr d\theta<\infty\\
\end{align*}
Thus $(z_1-1)^{-\beta_j}\in A^2(B(0,1))$ for all $j\in \mathbb{N}$.\\

By convexity of $B(0,1)$, there exists $s>0$, $\theta_1,\theta_2\in [0,2\pi]$, $\theta_2>\theta_1$ so that \[\{z=1+r_1e^{i\theta}:0\leq r_1\leq s,\, \theta_1\leq \theta\leq \theta_2\}\times \{w=r_2e^{i\theta}:0\leq r_2\leq s,\, \theta_1\leq \theta\leq \theta_2\}\subset B(0,1)\]
Using this inclusion and converting to polar coordinates, one can show that
$\|\frac{1}{(z_1-1)^{\beta_j}}\|\rightarrow \infty$ as $j\rightarrow \infty$.\\  

Then we define $\alpha_j:=\|\frac{1}{(z_1-1)^{\beta_j}}\|_{B(0,1)}^{-1}$.
So $\|f_j\|=1$ for all $j\in \mathbb{N}$ and $f_j\rightarrow 0$ uniformly on compact subsets of $B(0,1)$ away from $bB(0,1)$.  Thus by \cite[lemma 3.5]{cuckruhanseveral}, $f_j\rightarrow 0$ weakly as $j\rightarrow \infty$.  However,
$f_j\circ \phi(z_1,z_2)=\frac{\alpha_j}{(z_1-1)^{\beta_j}}$ does not converge to $0$ in norm.  Therefore, $C_{\phi}$ is not compact.

\section{Aknowlegments}
I wish to thank Kit Chan and S\"{o}nmez \c{S}ahuto\u{g}lu for useful conversations and insights.  I also thank Joe Cima for commenting on a preliminary version of this manuscript.

\bibliographystyle{amsalpha}
\bibliography{refscomp}

\end{document}